\documentclass[reqno]{amsart}
\usepackage{amssymb}  
\usepackage{amsmath} 
\usepackage{amsthm}  
\usepackage{bm} 
\usepackage{stmaryrd} 
\usepackage{floatflt} 

\advance\hoffset by -1in
\advance\voffset by -1in

\newtheorem{thm}{Theorem}
\newtheorem{lem}[thm]{Lemma}

\newtheorem{prob}{Problem}

\title[The structure of maximal tori]{The structure of maximal tori \\ in spin groups of type $D_l$}
\author{Andrei V. Zavarnitsine}
\address{\textup{\scriptsize
Andrei V. Zavarnitsine\\
Group Theory Lab.\\
Sobolev Institute of Mathematics\\
4, Koptyug av.\\
630090, Novosibirsk, Russia\\
\texttt{and} Mechanics and Mathematics Dept.\\
Novosibirsk State University\\
2, Pirogova st.\\
630090, Novosibirsk, Russia
}}
\email{zav@math.nsc.ru}

\thanks{Supported by the RScF (project 14--21--00065)}
\date{}

\begin{document}
\begin{abstract} We determine the abelian invariants of the maximal tori in the finite spin groups of type $D_l$.

{\sc Keywords:} Spin groups, maximal tori

{\sc MSC2010:}  20E34,  
                20G15   

{\sc UDC:} 512.54 
\end{abstract}
\maketitle

\section{Introduction}

The maximal tori in finite groups of Lie type arise as fixed points of Frobenius endomorphisms acting on the corresponding maximal tori of the ambient algebraic groups. Being abelian, the finite maximal tori decompose into a direct product of cyclic groups. An explicit form of this decomposition has been determined for many classical and Chevalley groups, see, e.\,g. \cite{07ButGr.t, 91DerFak}. This information is important, for example, in the study of the set of element orders and some problems of representation theory. In this paper, our aim is to find a similar cyclic decomposition for the tori in the groups $\operatorname{Spin}_{2l}^\pm(q)$.

Let $F$ be an algebraically closed field of positive characteristic $p$ and let $\mathbb{G}$ be a simply-connected simple linear algebraic group of type $D_l$ over $F$. It is well known that $\mathbb{G}\cong \operatorname{Spin}_{2l}(F)$.
Let $\sigma$ be a Frobenius endomorphism of $\mathbb{G}$ associated with the field automorphism $[t\mapsto t^q]$ of $F$, where $q=p^m$, and a symmetry of the Dynkin diagram of type $D_l$ of order $1$ or $2$. Then $G=C_{\mathbb{G}}(\sigma)$ is isomorphic to the finite spin group $\operatorname{Spin}_{2l}^\pm(q)$, where the sign is $+$ or $-$ according as the symmetry has order $1$ or $2$. If $\mathbb{T}$ is a $\sigma $-invariant maximal torus of $\mathbb{G}$ then we say that $T=C_\mathbb{T}(\sigma)$ is a maximal torus of the finite group $G$.

The $G$-conjugacy classes of $\sigma$-invariant maximal tori of $\mathbb{G}$ are in a one-to-one correspondence with
the $\sigma$-conjugacy classes of the Weyl group $W=W(D_l)$. The corresponding tori $T$ are thus parameterized by signed partitions $l=l'+l''$, with $l'=l_1+\ldots+l_r$ and $l''=l_{r+1}+\ldots+l_{r+s}$ corresponding to positive and negative terms. Such a partition parameterizes a torus in $\operatorname{Spin}_{2l}^+(q)$ if $s$ is even, and in $\operatorname{Spin}_{2l}^-(q)$ if $s$ is odd, whose structure can be described uniformly in both cases.

Denote $L'=\{l_1,\ldots,l_r\}$ and $L''=\{l_{r+1},\ldots,l_{r+s}\}$. Also, set $\varepsilon_i=1$ if $1\leqslant i\leqslant r$ and $\varepsilon_i=-1$ if $r+1\leqslant i\leqslant r+s$. Let $\mathbb{Z}_n$ denote a cyclic group of order $n$. We prove

\begin{thm}\label{main} Let $G\cong \operatorname{Spin}_{2l}^\pm(q)$ and let $T$ be a maximal torus of $G$ corresponding to the signed partition of $l$ as above.
\begin{enumerate}
\item[$(i)$] If there are odd elements $l_i\in L'$ and $l_j\in L''$ then
$$
T\cong \mathbb{Z}_{(q^{l_i}-1)(q^{l_j}+1)}\times \prod_{k\ne i,j}\mathbb{Z}_{q^{l_k}-\varepsilon_k}
$$
\item[$(ii)$] If there is odd $l_i$ in either $L'$ or $L''$ but not both, and there is even $l_j$ in $L''$ then
$$
T\cong \mathbb{Z}_{(q^{l_i}-\varepsilon_i)(q^{l_j}+1)}\times \prod_{k\ne i,j}\mathbb{Z}_{q^{l_k}-\varepsilon_k}
$$
\item[$(iii)$] If all elements of $L'$ are even and $L''=\varnothing$ then choose $l_i\in L'$ with minimal $2$-part. In this case,
$$
T\cong \mathbb{Z}_{q^{l_i/2}-1}\times \mathbb{Z}_{q^{l_i/2}+1}\times \prod_{k\ne i}\mathbb{Z}_{q^{l_k}-1}
$$
\item[$(iv)$] In all other cases,
$$T\cong \prod_{k}\mathbb{Z}_{q^{l_k}-\varepsilon_k}$$
\end{enumerate}
$($In all products above, $1\leqslant k \leqslant r+s$.$)$
\end{thm}

We remark that in case $(iii)$, which only occurs for  $\operatorname{Spin}_{2l}^+(q)$, $l$ even, there are two $\sigma$-conjugacy classes of the Weyl group associated with each partition of $l$, but the corresponding tori are isomorphic as we show below.
Also, note that we make no assumptions on $q$ and $l$ in the theorem, but for small $l$ the structure of maximal tori in  $\operatorname{Spin}_{2l}^\pm(q)$ might as well be derived from known results due to the exceptional isomorphisms
\begin{align*}
\operatorname{Spin}_{2}^\pm(q)\cong \mathbb{Z}_{q\mp 1},& \qquad \operatorname{Spin}_{4}^+(q)\cong \operatorname{SL}_2(q)\times \operatorname{SL}_2(q), \qquad \operatorname{Spin}_{4}^-(q)\cong \operatorname{SL}_2(q^2),\\
&\operatorname{Spin}_{6}^+(q)\cong \operatorname{SL}_4(q), \qquad \operatorname{Spin}_{6}^-(q)\cong \operatorname{SU}_4(q).
\end{align*}
Similarly, for $q$ even,  $\operatorname{Spin}_{2l}^\pm(q)$ is isomorphic to $\Omega_{2l}^\pm(q)$ whose maximal tori are known  \cite[Theorem 7]{07ButGr.t}. In this case, all tori fully decompose as in $(iv)$.

If $q$ is odd and, under the assumptions of case $(i)$, there is also an even $l_k$ in $L''$, then the torus $T$ also admits the decomposition as in $(ii)$, namely
$$
T\cong \mathbb{Z}_{(q^{l_t}-\varepsilon_t)(q^{l_k}+1)}\times \prod_{n\ne k,t}\mathbb{Z}_{q^{l_n}-\varepsilon_n},
$$
where $t\in\{i,j\}$ is such that $q\equiv \varepsilon_t\pmod 4$. This follows from Lemma \ref{nt}$(vi)$ below.

Another remark is that due to the duality of algebraic groups, Theorem \ref{main} can also be used to determine the structure of maximal tori in finite adjoint groups of type $D_l$ which are isomorphic to $PCO_{2l}^\pm(q)^\circ$, for details see \cite[Sections 1.19, 4.4]{85Car}.

As an example, we list explicitly in Table \ref{sp8} below the structure of maximal tori in $\operatorname{Spin}_8^\pm(q)$.
The following problem remains open.
\begin{prob} Determine the cyclic structure of the maximal tori in the finite spin
groups of type $B_l$.
\end{prob}

\section{Preliminaries}

\begin{floatingfigure}[l]{35mm}
\setlength{\unitlength}{.06mm}
\begin{picture}(485,650)(20,0)

 \put (250,490){\line(0,1){120}}
 \put (251,490){\line(0,1){120}}
 \put (252,490){\line(0,1){120}}

 \put (250,250){\line(0,1){120}}
 \put (250,250){\line(2,-1){120}}
 \put (250,250){\line(-2,-1){120}}

 \put (251,250){\line(0,1){120}}
 \put (251,250){\line(2,-1){120}}
 \put (251,250){\line(-2,-1){120}}

 \put (252,250){\line(0,1){120}}
 \put (252,250){\line(2,-1){120}}
 \put (252,250){\line(-2,-1){120}}

 \put (253,250){\line(2,-1){120}}
 \put (249,250){\line(-2,-1){120}}

 \put (251,610){\circle*{22}} \put (282,615){$\alpha_1$}
 \put (251,490){\circle*{22}} \put (282,495){$\alpha_2$}

 \put (251,410){\circle*{2}}
 \put (251,430){\circle*{2}}
 \put (251,450){\circle*{2}}

 \put (251,370){\circle*{22}} \put (282,375){$\alpha_{l-3}$}
 \put (251,250){\circle*{22}} \put (282,267){$\alpha_{l-2}$}
 \put (131,190){\circle*{22}} \put (55,140) {$\alpha_{l-1}$}
 \put (371,190){\circle*{22}} \put (395,140){$\alpha_l$}

 \thinlines
 \put (280,188){\vector(1,0){60}}
 \put (220,188){\vector(-1,0){60}}\put (240,178){$\scriptstyle\rho$}
 \put (105,40){Figure 1. $D_l$}
\end{picture}
\end{floatingfigure}

We recall that the spin groups are classically defined via Clifford algebras (see, e.\,g., \cite[Section 4.8]{89Jac}). However, we will use their identification with universal Chevalley groups, which is more suitable for our purposes, see \cite{72Car,85Car,98GorLySol}.

As above, let $\mathbb{G}$ be a simply-connected simple linear algebraic group of type $D_l$ over an algebraically closed field $F$ of positive characteristic $p$. We will assume that $l\geqslant2$.  Denote by $\Phi$ the root system of type $D_l$ with the set $\Pi=\{\alpha_1,\ldots,\alpha_l\}$ of fundamental roots and Dynkin diagram shown in Fig. 1. We may view $\mathbb{G}$ as a Chevalley group generated by the symbols $x_\alpha(t)$, $t\in F$
and $\alpha \in \Phi$, that satisfy the Chevalley relations. It is known that $\mathbb{G}\cong\operatorname{Spin}_{2l}(F)$.
For details, see \cite[Theorems 1.10.7, 1.12.1]{98GorLySol}

For $q=p^m$, let $\varphi_q$ be the field automorphism of $\mathbb{G}$ that acts on the generators by $\varphi_q:x_\alpha(t)\mapsto x_\alpha(t^q)$, $\alpha\in \Phi$, $t\in F$,  and let $\gamma_\rho$ be the graph automorphism of $\mathbb{G}$ associated with the two-fold symmetry $\rho: \alpha_{l-1} \leftrightarrow \alpha_l$, $\rho: \alpha_i\mapsto \alpha_i$, $i<l-1$, of the Dynkin diagram that acts on the generators by $\gamma_\rho: x_\alpha(t)\mapsto x_{\rho(\alpha)}(t)$, $\alpha\in \pm\Pi$, $t\in F$. We let $\sigma$ be either $\varphi_q$ or $\varphi_q\gamma_\rho$ and then
$G=C_{\mathbb{G}}(\sigma)$ will be isomorphic to either $\operatorname{Spin}_{2l}^+(q)$ or $\operatorname{Spin}_{2l}^-(q)$, respectively. There is a short exact sequence
$$
1\to Z \to \operatorname{Spin} ^\varepsilon_{2l}(q) \to \mathrm{P}\Omega_{2l}^\varepsilon(q)\to 1,
$$
where $\varepsilon=\pm$, and $Z=\mathrm{Z}(\operatorname{Spin} ^\varepsilon_{2l}(q))$ is isomorphic to $\mathbb{Z}_{(2,q-1)}^2$ if $\varepsilon=+$ and $l$ is even, and to $\mathbb{Z}_{(4,q^l-\varepsilon)}$, otherwise.

Let $\mathbb{T}$ be a $\sigma$-invariant maximal torus of $\mathbb{G}$. Explicitly, we may choose $\mathbb{T}$ to be generated by the elements $h_\alpha(t)=n_\alpha(t)n_\alpha(-1)$, where $n_\alpha(t)=x_\alpha(t)x_{-\alpha}(t^{-1})x_\alpha(t)$ for $\alpha\in \Phi$, $t\in F^\times$.  For every  $\alpha\in \Phi$, the root subgroup $\langle x_\alpha(t)\mid t\in F^\times \rangle$ is $\mathbb{T}$-invariant. In particular, $x_\alpha(t)^\tau=x_\alpha(r_\alpha(\tau)t)$ for all $\tau\in \mathbb{T}$, $t\in F^\times$, and a suitable element $r_\alpha$ of the character group $X=\operatorname{Hom}(\mathbb{T},F^\times)$ which we may identify with $\alpha$.
The fundamental roots $\alpha_1,\ldots,\alpha_l$ form an $\mathbb{R}$-basis of $X\otimes \mathbb{R}$.
Similarly, the map $h_\alpha$ lies in the group of cocharacters $Y=\operatorname{Hom}(F^\times,\mathbb{T})$ and may be identified with the coroot $\alpha^\vee$ corresponding to $\alpha$. There is a natural bilinear pairing $X\times Y\to \mathbb{Z}$, $(\chi,\theta)\mapsto \langle\chi,\theta\rangle$ such that $\langle\alpha,\alpha^\vee\rangle=2$ for all $\alpha\in \Phi$.
Since $\mathbb{G}$ is simply-connected, the coroots $\alpha_1^\vee,\ldots,\alpha_l^\vee$ form a $\mathbb{Z}$-basis of $Y$. The dual $\mathbb{Z}$-basis $\omega_1,\ldots,\omega_l$ of $X$, i.\,e. such that $\langle\omega_i,\alpha_j^\vee\rangle=\delta_{ij}$, $i,j=1,\ldots,l$, is called the set of {\em fundamental weights}.

The Weyl group $W=N_\mathbb{G}(\mathbb{T})/\mathbb{T}$ acts naturally on $\mathbb{T}$.
Since $\mathbb{T}$ is $\sigma$-invariant, there is an induced action of $\sigma$ on $W$ and a (left) action of both $\sigma$ and $W$
on $X$ by $^u\chi(\tau)=\chi(\tau^u)$,  for all $\chi\in X$, $\tau\in\mathbb{T}$, $u\in W\leftthreetimes\langle\sigma\rangle$.
In the basis $\alpha_1,\ldots,\alpha_l$, the matrix of $\sigma$ is either $\operatorname{diag}(q,\ldots,q)$ if $\sigma=\varphi_q$ or
$M$ defined in (\ref{mats1}) if $\sigma=\varphi_q\gamma_\rho$.

Two elements $w,w'\in W$ are {\em $\sigma$-conjugate}, if there exists $x\in W$ such that $w'=x^{-1}wx^\sigma$.
This is an equivalence relation on $W$ whose classes are called the {\em $\sigma$-conjugacy classes}.
The map $w\mapsto w\sigma^{-1}$ induces an element-wise bijection between the $\sigma$-conjugacy classes of $W$
and ordinary $W$-conjugacy classes in the coset $W\sigma^{-1}$.

Let $\pi: N_\mathbb{G}(\mathbb{T})\to W$ be the natural epimorphism. If $g\in G$ is such that $^g\mathbb{T} = g\mathbb{T} g^{-1}$ is $\sigma$-invariant then $g^{-1}g^\sigma\in N_\mathbb{G}(\mathbb{T})$.

\begin{lem}\cite[Propositions 3.2.3, 3.3.3, 3.3.4]{72Car} \label{twt}
\begin{enumerate}
\item[$(i)$] The map $^g\mathbb{T}\to \pi(g^{-1}g^\sigma)$ determines a bijection between the $G$-conjugacy classes of $\sigma$-invariant maximal tori of $\mathbb{G}$ and the $\sigma$-conjugacy classes of $W$.
\item[$(ii)$] If $^g\mathbb{T}$ is $\sigma$-invariant then $C_{^g\mathbb{T}}(\sigma)$ is isomorphic to $X/X^{\sigma w^{-1}-1}$, where $w=\pi(g^{-1}g^\sigma)$.
\end{enumerate}
\end{lem}

The Euclidean space $X\otimes \mathbb{R}$ has an orthonormal basis $\nu_1,\ldots \nu_l$ such that
\begin{equation}\label{trans}
\alpha_i=\nu_i-\nu_{i+1},\quad i=1,\ldots,l-1, \quad \alpha_l=\nu_{l-1}+\nu_l.
\end{equation}
The Weyl group $W$ is usually identified with the group generated by the reflections $w_{\alpha_i}$ of $X\otimes \mathbb{R}$ in the fundamental roots.
Written in the basis $\nu_1,\ldots \nu_l$, the group $W$ is faithfully represented by monomial $l\times l$-matrices
with nonzero entries $\pm 1$. Such matrices are in bijection with the {\em signed permutations}, i.\,e. elements $\theta$ of the group $Sym(\{\pm 1,\ldots,\pm l\})$
satisfying $\theta(-i)=-\theta(i)$, $i=1,\ldots,l$, which is isomorphic to the Weyl group $W(C_l)$.
Every signed permutation $\theta$ is a unique product of positive and negative cycles of the form $(i_1,\ldots,i_k)$ and $\overline{(i_1,\ldots,i_k)}$ which represent the permutations
\begin{equation}\label{sp}
i_1\to \ldots \to i_k\to i_1, \qquad i_1\to \ldots \to i_k\to -i_1,
\end{equation}
respectively. A a signed partition $l=l'+l''$ with positive part $l'=l_1+\ldots+l_r$ and negative part $l''=l_{r+1}+\ldots+l_{r+s}$ determines a {\em signed-cycle type} $[l_1,\ldots,l_r,\overline l_{r+1},\ldots,\overline l_{r+s}]$ of $\theta$ if the signed-cycle decomposition of $\theta$ contains positive cycles of lengths $l_i$, $1\leqslant i\leqslant s$, and negative cycles of lengths $l_i$, $r+1 \leqslant i \leqslant r+s$.

The Weyl group $W=W(D_l)$ is a subgroup of $W(C_l)$ of index $2$ consisting of the elements with an even number of negative cycles.

\begin{lem}\cite[Propositions 24,25]{72CarC}\label{conjw}
Two element of $W(C_l)$ are conjugate if and only if they have the same signed-cycle type. A conjugacy class of $W(C_l)$ that lies in $W$ remains a complete $W$-conjugacy class, except when
it consists of elements with all cycles positive of even length, in which case it splits in two $W$-classes.
\end{lem}

As a standard representative of the $W(C_l)$-class of the signed-cycle type $[l_1,\ldots,\allowbreak\overline l_{r+s}]$, we choose the permutation
$$
(1,\ldots,l_1)(l_1+1,\ldots,l_1+l_2)\ldots \overline{(l-l_{r+s},\ldots,l)}
$$
If all $l_i$ are even and $l''=0$ then the $W(C_l)$-class of type $[l_1,\ldots, l_r]$ with standard representative $\theta=(1,\ldots,l_1)\ldots(\ldots,l)$ splits in two $W$-classes of types $[l_1,\ldots, l_r]^\pm$ with standard representatives $\theta^+=\theta$ and $\theta^-=(1,\ldots,l_1)\ldots(\ldots,l-1,-l)$. In fact, $\theta^+$ and $\theta^-$ are conjugate by the involution $\overline{(l)}\in W(C_l)\setminus W$.

The above suggests a natural way of identifying the matrices of $\sigma w^{-1}$ in the $\mathbb{Z}$-basis $\omega_1,\ldots,\allowbreak\omega_l$ of $X$, where $w$ runs through the representatives of $\sigma$-conjugacy classes of $W$,  with the matrices of standard representatives of signed permutation. (Indeed, by (\ref{trans}) and (\ref{mats1}), the $\sigma$-conjugacy of $W$ is the ordinary conjugacy if $\sigma=\varphi_q$ and the $\overline{(l)}$-conjugacy if $\sigma=\varphi_q\gamma_\rho$, and by \cite[p. 45]{72CarC}, every element of a Weyl group is conjugate to its inverse.)

We will need some number-theoretic facts. For a natural number $n$ denote by $n_2$ the $2$-part of $n$, i.\,e. the largest power of $2$ dividing $n$. The GCD of $a$ and $b$ is $(a,b)$.

\begin{lem}\label{a2}
For $a,b$ natural numbers, if $a_2\geqslant b_2$ then $(2a,b)=(a,b)$.
\end{lem}
\begin{proof} We have $(a,b)_2=\operatorname{min}\{a_2,b_2\}=b_2=\operatorname{min}\{2a_2,b_2\}=(2a,b)_2$, since $a_2\geqslant b_2$. The claim follows.
\end{proof}

\begin{lem}\label{nt}
Suppose $a$ is an odd natural number.

\begin{enumerate}
\item[$(i)$] For every natural $n$,
$$
\begin{array}{r@{\,}l}
(a^n-1)_2=&\left\{
\begin{array}{ll}
n_2(a+1)_2, &\text{if}\ \ n\ \ \text{is even and}\ \ a\equiv -1\pmod{4};\\
n_2(a-1)_2, &\text{otherwise};
\end{array}
\right.\\[10pt]
(a^n+1)_2=&\left\{
\begin{array}{ll}
(a+1)_2,\quad &\text{if}\ \ n\ \ \text{is odd};\\
2, &\text{if}\ \ n\ \ \text{is even}.
\end{array}
\right.
\end{array}
$$

\item[$(ii)$] If $n_1,n_2$ are odd and $\varepsilon=\pm1$ then $(a^{n_1}-\varepsilon,a^{n_2}+\varepsilon)=2$.

\item[$(iii)$] If $n_1$ is even and $n_2$ is odd then $(a^{n_1}+1,a^{n_2}\pm1)=2$.

\item[$(iv)$] If $n_1,n_2$ are odd and $\varepsilon=\pm1$ then $(a^{n_1}+\varepsilon,a^{n_2}+\varepsilon)=a^{(n_1,n_2)}+\varepsilon$.

\item[$(v)$] If $n_1$ is even, $n_2$ is odd, and $\varepsilon=\pm1$ then $(a^{n_1}-1,a^{n_2}+\varepsilon)=a^{(n_1,n_2)}+\varepsilon$.

\item[$(vi)$] If $n_1,n_2$ are odd, $n_3$ is even, and $a\equiv\varepsilon\pmod 4$, where $\varepsilon=\pm 1$, then
$$
\mathbb{Z}_{(a^{n_1}+\varepsilon)(a^{n_2}-\varepsilon)}\times \mathbb{Z}_{a^{n_3}+1} \cong \mathbb{Z}_{a^{n_1}+\varepsilon}\times \mathbb{Z}_{(a^{n_2}-\varepsilon)(a^{n_3}+1)}
$$

\end{enumerate}
\end{lem}
\begin{proof} $(i)$ See \cite[Lemma 8]{08Gr.t}. $(ii)$--$(v)$ See \cite[Lemma 6]{04Zav}.

$(vi)$ By $(ii)$,$(iii)$ the $2'$-parts of the numbers $a^{n_1}+\varepsilon$, $a^{n_2}-\varepsilon$, $a^{n_3}+1$ are pairwise coprime.
Since, by assumption, $(a^{n_1}+\varepsilon)_2=(a^{n_3}+1)_2=2$, the claim follows.
\end{proof}

In all matrices below, the dot entries stand for zeros.

\begin{lem} \label{mt} Let $a,b,c$ be integers.
\begin{enumerate}
\item[$(i)$] If $(a,b)=1$ and
$$
A=
\left(
  \begin{array}{cc}
    a & . \\
    . & b
  \end{array}
\right),
\qquad
B=
\left(
  \begin{array}{cc}
    1 & . \\
    . & ab
  \end{array}
\right),
$$
then set
\begin{equation} \label{pq1}
P=\left(
  \begin{array}{cc}
    1 & n \\
    -b & am
  \end{array}
\right),
\qquad
Q=\left(
  \begin{array}{cc}
    m & -bn \\
    1 & a
  \end{array}
\right),
\end{equation}
where $m,n$ are such that $am+bn=1$.

\item[$(ii)$] If
$$
A=
\left(
  \begin{array}{cc}
    a & 1 \\
    . & b
  \end{array}
\right),
\qquad
B=
\left(
  \begin{array}{cc}
    1 & . \\
    . & ab
  \end{array}
\right),
$$
then set
\begin{equation}\label{pq2}
P=\left(
  \begin{array}{cc}
    1 & . \\
    -b & 1
  \end{array}
\right),
\qquad
Q=\left(
  \begin{array}{cc}
    . & -1 \\
    1 & a
  \end{array}
\right).
\end{equation}

\item[$(iii)$] If $(a,b)$ divides $c$ and
$$
A=
\left(
  \begin{array}{cc}
    a & c \\
    . & b
  \end{array}
\right),
\qquad
B=
\left(
  \begin{array}{cc}
    a & . \\
    . & b
  \end{array}
\right),
$$
then set
\begin{equation}\label{pq3}
P=\left(
  \begin{array}{cc}
    1 & -nz \\
    . & 1
  \end{array}
\right),
\quad
Q=\left(
  \begin{array}{cc}
    1 & -mz \\
    . & 1
  \end{array}
\right),
\end{equation}
where $x=a/(a,b)$, $y=b/(a,b)$, $z=c/(a,b)$ and $m,n$ satisfy $xm+yn=1$.

\item[$(iv)$] If $(a,b)=1$, $c$ is odd, and
$$
A=
\left(
  \begin{array}{cc}
    2a & c \\
    . & 2b
  \end{array}
\right),
\qquad
B=
\left(
  \begin{array}{cc}
    1 & . \\
    . & 4ab
  \end{array}
\right),
$$
then set
\begin{equation}\label{pq4}
P=\left(
  \begin{array}{cc}
    -1 & n(c-1)/2 \\
    -2b & 1+nb(c-1)
  \end{array}
\right),
\quad
Q=\left(
  \begin{array}{cc}
    m(c-1)/2 & -1-ma(c-1) \\
    -1 & 2a
  \end{array}
\right),
\end{equation}
where $m,n$ are such that $am+bn=1$.
\end{enumerate}

Then, in all four cases above, we have $P,Q\in \operatorname{GL}_2(\mathbb{Z})$ and $PAQ=B$.
\end{lem}
\begin{proof} Straightforward.
\end{proof}

For a matrix $M$, we denote by $M^\top$ the transpose of $M$.
The following matrices are referred to from elsewhere in the paper.

\begin{equation}\label{mats1}
M=
\left(
\begin{array}{ccccc}
q \\
 &\ddots\\
 &       &  q\\
 &       &  & . &  q \\
 &       &  & q & .
\end{array}
\right)
\qquad
S=
\left(
\begin{array}{rrrrr}
 1 & -1 & . \\
 . & 1 & -1 \\
 &  & \ddots  &\ddots \\
 &  &    & 1  &  -1 \\
 &  &    & 1  &  1 \\
\end{array}
\right)_{l\times l}
\end{equation}

\begin{equation}\label{mats2}
R_{\varepsilon k}=
\left(
\begin{array}{ccccc}
 . & 1 & . \\
 . & . & 1\\
 &  & & \ddots \\

 .  & . & .    &  &  1 \\
 \varepsilon & . & .   &  &  . \\
\end{array}
\right)_{k\times k}
\quad
B_i=
\left(
\begin{array}{cc@{}c@{}c@{\ \ }c}
 \varepsilon_{-1} & . &  & . & -(1+\varepsilon_{-1})/2\\
 \varepsilon_{-1} & . &  & . & -(1+\varepsilon_{-1})/2\\
   & & \ldots \\
 \varepsilon_{-1} & . &  & . & -(1+\varepsilon_{-1})/2\\
 \varepsilon_{-1} & . &  & . & -(\varepsilon_i+\varepsilon_{-1})/2\\
\end{array}
\right)_{l_i\times l_{-1}}
\end{equation}

\begin{equation}\label{mats3}
J=
\left(
\begin{array}{ccccc}
 . & . &  & . & 1\\
 . & . &  & . & 1\\
   & & \ldots \\
 . & . &  & . & 1\\
\end{array}
\right)_{l\times l}
\quad
P_{\varepsilon k}=
\left(
\begin{array}{llllll}
 1 & q & q^2& & q^{k-2} & . \\
 . & 1 & q  & & q^{k-3} & . \\
 . & . & 1  & & q^{k-4} & . \\
   &   &    & \ddots \\

 . & . & .  &  & 1 & . \\
 q &q^2& q^3& & q^{k-1} & \varepsilon \\
\end{array}
\right)_{k\times k}
\end{equation}

Note that the following degenerate cases may occur
\begin{align*}
&R_{\varepsilon 1}=P_{\varepsilon 1}=(\varepsilon),\\
&B_i=(-(1-\varepsilon_{-1})/2,\ldots, -(1-\varepsilon_{-1})/2,-(\varepsilon_i-\varepsilon_{-1})/2 )^\top,\ \ \text{if}\ \ l_{-1}=1,\\
&B_i=( \varepsilon_{-1}, 0,\ldots, 0, -(\varepsilon_i+\varepsilon_{-1})/2),\ \ \text{if}\ \ l_{i}=1, \\
&B_i=( -(\varepsilon_i-\varepsilon_{-1})/2),\ \ \text{if}\ \ l_{-1}=l_{i}=1.
\end{align*}

\section{Proof of main theorem}

We fix a signed-cycle type $[l_1,\ldots,\overline l_{r+s}]$ and set $\varepsilon_i=1$ if $1\leqslant i\leqslant r$ and $\varepsilon_i=-1$ if $r+1\leqslant i\leqslant r+s$. In the basis $\nu_1,\ldots,\nu_l$, the matrix of
the corresponding standard representative $\theta$ is $R=\bigoplus R_{\varepsilon_i l_i}$, where
$R_{\varepsilon k}$ is the matrix (\ref{mats2}) of the signed cycles (\ref{sp}) for $\varepsilon=\pm1$, respectively.
(In the exceptional case $\theta=\theta^-$, the matrix is given in Lemma \ref{ex}.)
Our aim is to diagonalize the integer matrix
\begin{equation}\label{orm}
qSRS^{-1}-E,
\end{equation}
corresponding  to the transformation $\sigma w^{-1}-1$,
where $E$ is the identity matrix and $S$ is the transition matrix (\ref{mats1}) to the basis $\omega_1,\ldots,\omega_l$, which by Lemma \ref{twt} will give us the cyclic structure of the maximal torus of $\operatorname{Spin}_{2l}^\pm(q)$ that corresponds to the chosen signed-cycle type. We first dispose of the exceptional case.

\begin{lem} \label{ex}
If $s=0$ and all $l_i$ are even then the maximal tori of $\operatorname{Spin}^+_{2l}(q)$ that correspond to the two signed-cycle types $[l_1,\ldots, l_r]^\pm$ are isomorphic.
\end{lem}
\begin{proof} Since $\theta^+$ and $\theta^-$ are conjugate by $\overline{(l)}$, the corresponding matrices are $qSRS^{-1}-E$
and $qSR_0RR_0^{-1}S^{-1}-E$, where $R_0=\operatorname{diag}(1,\ldots,1,-1)$ is the permutation matrix of $\overline{(l)}$. Since these matrices are conjugate by $SR_0S^{-1}$ which is integral and unimodular, the claim follows.
\end{proof}

Two integer $l\times l$-matrices $M_1$ and $M_2$ are said to be {\em equivalent}, if there are $P,Q\in\operatorname{GL}_l(\mathbb{Z})$ such that $PM_1Q=M_2$. Since the multiplication by $P$ (respectively, $Q$) amounts to performing elementary row (column) transformations of $M_1$, in case $Q=E$ (respectively, $P=E$), we say that $M_1$ and $M_2$ are {\em row-} ({\em column-}) {\em equivalent}.

Since $S$ is row-equivalent to $E+J$, where $J$ is given in (\ref{mats3}), we see that (\ref{orm}) is equivalent to
\begin{equation}\label{orm1}
q(E+J)R(E-\frac{1}{2}J)-E=qR-E+qB,
\end{equation}
where $qR-E$ is block diagonal and $B$ is zero except for the last block-column with entries $B_1,\ldots,B_{r+s}$. ($B_i$ is defined in (\ref{mats2}), where we have abbreviated $\varepsilon_{-i}=\varepsilon_{r+s+1-i}$ and $l_{-i}=l_{r+s+1-i}$ for $i=1,2,\ldots$)
Multiplying (\ref{orm1}) on the left by the block diagonal matrix $\operatorname{diag}(P_{\varepsilon_1 l_1},\ldots, P_{\varepsilon_{-2} l_{-2}}, P'_{\varepsilon_{-1} l_{-1}})$, where $P_{\varepsilon k}$ is defined in (\ref{mats3}) and
$P'_{\varepsilon_{-1} l_{-1}}$ is obtained from $P_{\varepsilon_{-1} l_{-1}}$ by substituting the last row with $(0,\ldots,0,\varepsilon_{-1})$, we arrive at a matrix with some columns having $-1$ on the diagonal and zeros elsewhere.
Using them to annihilate nonzero row entries by column transformations, we obtain (after deleting the identity rows and columns) the matrices

\begin{equation}\label{orm2}
\left(
\begin{array}{rrrrr}
              \multicolumn{3}{c|}{}                    & a_1 &  b_1 \\
               & D &   \multicolumn{1}{c|}{}           & \vdots & \vdots   \\
              \multicolumn{3}{c|}{}            & a_{-2} & b_{-2} \\
\cline{1-3}
       .       &  \ldots      &                   . & a & b \\
       .       &  \ldots      &                   . & 2q & -q-\varepsilon_{-1}
\end{array}
\right)\quad\text{or}\quad
\left(
\begin{array}{rrrr}
     \multicolumn{3}{c|}{}   & a_1+ b_1 \\
               & D      & \multicolumn{1}{c|}{}   & \vdots    \\
     \multicolumn{3}{c|}{}  & a_{-2}+ b_{-2} \\
\cline{1-3}
       .       &  \ldots      &                   . & q-\varepsilon_{-1}
\end{array}
\right)
\end{equation}
according as $l_{-1}>1$ or $l_{-1}=1$, where
\begin{align*}
D &= \operatorname{diag}(q^{l_1}-\varepsilon_1,\ldots,q^{l_{-2}}-\varepsilon_{-2}),\\
a_i &= \varepsilon_{-1}(\varepsilon_i q+q^2+\ldots+q^{l_i}), \quad i=1,2,\ldots\\
b_i &= -((1+\varepsilon_{-1}\varepsilon_i)q+(1+\varepsilon_{-1})(q^2+\ldots+q^{l_i}))/2, \quad i=1,2,\ldots\\
 a  &= -1+\varepsilon_{-1}(q+q^2+\ldots+q^{l_{-1}-1}), \\
 b  &= -(1+\varepsilon_{-1})(q+\ldots+q^{l_{-1}-1})/2-q^{l_{-1}-1}.
\end{align*}

As we have mentioned in the introduction, the theorem holds for $q$ even. Hence, we assume from now on that $q$ is odd.

If $l_{-1}=1$ then, for the second matrix in (\ref{orm2}), we may add the last row multiplied by
$$
\left\{
\begin{array}{ll}
((1-\varepsilon_{-1})(q+q^3+\ldots+q^{l_i-1})+\varepsilon_{-1}(1-\varepsilon_i))/2 &\text{if}\quad l_i \ \ \text{is even}, \\
((1-\varepsilon_{-1})(q^2+q^4+\ldots+q^{l_i-1})+1-\varepsilon_{-1}\varepsilon_i)/2 &\text{if}\quad l_i \ \ \text{is odd},
\end{array}
\right.
$$
to the $i$th row, $i=1,2,\ldots$, to obtain

\begin{equation}\label{orm3}
\left(
\begin{array}{rrrr}
     \multicolumn{3}{c|}{}   & c_1 \\
               & D      & \multicolumn{1}{c|}{}   & \vdots    \\
     \multicolumn{3}{c|}{}  & c_{-2} \\
\cline{1-3}
       .       &  \ldots      &                   . & q-\varepsilon_{-1}
\end{array}
\right),
\end{equation}
where $c_i=(\varepsilon_i-1)/2$ if $l_i$ is even and $(\varepsilon_i-\varepsilon_{-1})/2$ if $l_i$ is odd. We will show below that a subcase of the
general case $l_{-1}>1$ reduces to the same matrix and hence may be treated together with the current case.

Suppose $l_{-1}>1$. The first matrix in (\ref{orm2}) when multiplied on the right by
$$
\left(
\begin{array}{rrrrr}
              \multicolumn{3}{c|}{}                    & (q-1)/2 &  . \\
               & E &   \multicolumn{1}{c|}{}           & \vdots & \vdots   \\
              \multicolumn{3}{c|}{}            & (q-1)/2 & . \\
\cline{1-3}
       .       &  \ldots      &                   . & (q+\varepsilon_{-1})/2 & -1 \\
       .       &  \ldots      &                   . & q & -2
\end{array}
\right)
$$
becomes a matrix whose last row $(0,\ldots,0,2\varepsilon_{-1})$ may be used to bring it by row transformations to the upper triangular form
\begin{equation}\label{orm4}
\left(
\begin{array}{rrrr}
            D &\multicolumn{1}{c|}{}    & (q^{l_1}-\varepsilon_1)/2 &  d_1 \\
              \multicolumn{2}{c|}{}     & \vdots & \vdots   \\
\cline{1-2}
       .  &  \ldots       & (q^{l_{-1}}-\varepsilon_{-1})/2 & d_{-1} \\
       .  &  \ldots       & . & 2
\end{array}
\right),
\end{equation}
where $d_i=0$ if $l_i$ is even and $1$ if $l_i$ is odd.

First, suppose that all $l_i$ are even and all $\varepsilon_i$ equal $1$. Without loss of generality we may assume
that $l_{-1}$ has the smallest $2$-part. If $r+s>1$, we multiply (\ref{orm4}) on the left and right by
$\operatorname{diag}(1,\ldots,1,P,1)$ and $\operatorname{diag}(1,\ldots,1,Q,1)$, respectively, where $P,Q$ are given by (\ref{pq3}) for $a=q^{l_{-2}}-1$, $b=(q^{l_{-1}}-1)/2$, and $c=a/2$. Note that $(a/2)_2\geqslant b_2$ by assumption and Lemma \ref{nt}$(i)$.
Therefore, Lemma \ref{a2} implies $(a,b)=(a/2,b)$ which divides $c$. By Lemma \ref{mt}$(iii)$ the resulting matrix will have lower right corner $\operatorname{diag}(a,b,2)$ and
all other entries intact, which follows from the explicit form of $Q$. Repeating, if necessary, this procedure, we may similarly annihilate all nonzero off-diagonal  entries of matrix (\ref{orm4}) thus bringing it to the equivalent diagonal form
$$
\operatorname{diag}(q^{l_1}-1,\ldots,q^{l_{-2}}-1,(q^{l_{-1}}-1)/2,2).
$$
It remains to observe that
$$
\mathbb{Z}_{(q^{l_{-1}}-1)/2}\times \mathbb{Z}_2\cong \mathbb{Z}_{q^{l_{-1}/2}-1}\times \mathbb{Z}_{q^{l_{-1}/2}+1},
$$
which yields item $(iii)$ of the theorem.

Hence, we may suppose that either there is an odd $l_i$ (and then we assume that $l_{-1}$ is odd), or all $l_i$ are even and, for some $i$, $\varepsilon_i=-1$ (and we assume that $i=-1$). In both cases, we multiply (\ref{orm4}) on the left and right by $\operatorname{diag}(1,\ldots,1,P)$ and $\operatorname{diag}(1,\ldots,1,Q)$, where $P$ and $Q$ are given by (\ref{pq2}) and (\ref{pq1}) in the former and latter cases, respectively, for $a=(q^{l_{-1}}-\varepsilon_{-1})/2$
and $b=2$. By lemma \ref{mt}, which is applicable in the latter case, since $((q^{l_{-1}}+1)/2,2)=1$, we obtain the following matrix (after removing the row and column with the diagonal pivot $1$)
\begin{equation}\label{orm5}
\left(
\begin{array}{rrrr}
     \multicolumn{3}{c|}{}   & f_1 \\
               & D      & \multicolumn{1}{c|}{}   & \vdots    \\
     \multicolumn{3}{c|}{}  & f_{-2} \\
\cline{1-3}
       .       &  \ldots      &                   . & q^{l_{-1}}-\varepsilon_{-1}
\end{array}
\right),
\end{equation}
where $f_i=d_i(q^{l_{-1}}-\varepsilon_{-1})/2-d(q^{l_i}-\varepsilon_i)/2$ and $d=1$ or $2n$ in the former and latter cases, respectively. We now consider these cases separately.

{\em Case I.} Assume that $l_{-1}$ is odd. Then $d=1$. Observe that if we set $l_{-1}=1$ then (\ref{orm5}) will be row equivalent to (\ref{orm3}), which may be seen after adding to the $i$th row the last row multiplied by $(q+\varepsilon_{-1})(q^{l_i-2}+q^{l_i-4}+\ldots+q^{d_i})/2$. We will therefore incorporate the above-postponed case $l_{-1}=1$ into further consideration.

{\em Case Ia.} Suppose that there is odd $l_j$ such that $\varepsilon_j=-\varepsilon_{-1}$. Without loss of generality we may assume that $j=-2$. In this case $f_{-2}=-a+b$ is odd, where $a=(q^{l_{-2}}+\varepsilon_{-1})/2$ and $b=(q^{l_{-1}}-\varepsilon_{-1})/2$. Observe also that $(a,b)=1$ by Lemma \ref{nt}$(ii)$. Hence, we may apply Lemma \ref{mt}$(iv)$, where  $c=-a+b$.  Multiplying (\ref{orm5}) by $\operatorname{diag}(1,\ldots,1,P)$ and $\operatorname{diag}(1,\ldots,1,Q)$, where $P$ and $Q$ are given in (\ref{pq4}), and then removing the row and column with the diagonal pivot $1$, we obtain the matrix
\begin{equation}\label{orm6}
\left(
\begin{array}{rrrr}
     q^{l_1}-\varepsilon_1          &  \multicolumn{2}{c|}{}   & g_1    \\
               & \ddots   & \multicolumn{1}{c|}{}   & \vdots    \\
    & &\multicolumn{1}{c|}{q^{l_{-3}}-\varepsilon_{-3}}   & g_{-3} \\
\cline{1-3}
      . &  \ldots      &                   . & (q^{l_{-2}}-\varepsilon_{-2})(q^{l_{-1}}-\varepsilon_{-1})
\end{array}
\right),
\end{equation}
where $g_i=2af_i=d_i\cdot2ab-a(q^{l_i}-\varepsilon_i)$. An obvious column transformation annihilates the summand $a(q^{l_i}-\varepsilon_i)$ and hence $g_i$, too, if $l_i$ is even. If $l_i$ is odd then $d_i=1$ and the submatrix
$$
\left(
\begin{array}{rr}
q^{l_i}-\varepsilon_i & 2ab \\
. & 4ab
\end{array}
\right)
$$
can be diagonalized by Lemma \ref{mt}$(iii)$ using the matrices $P$ and $Q$ given in (\ref{pq3}).
Observe that Lemma \ref{mt}$(iii)$ can be applied, since
$$
(2ab)_2=\frac{1}{2}(q^{l_{-2}}+\varepsilon_{-1})_2(q^{l_{-1}}-\varepsilon_{-1})_2=\frac{1}{2}(q^2-1)_2\geqslant (q-\varepsilon_i)_2=(q^{l_i}-\varepsilon_i)_2
$$
and $l_i$ is odd, which implies by Lemma \ref{a2} that $(4ab,q^{l_i}-\varepsilon_i)=(2ab,q^{l_i}-\varepsilon_i)$ and the latter divides $2ab$. Note that this diagonalization does not affect the remaining entries of (\ref{orm6}) and
therefore can be done independently for all nonzero $g_i$ thus bringing the matrix to a diagonal form. This proves item $(i)$ of the theorem.

{\em Case Ib.} Suppose that, for every odd $l_j$, we have $\varepsilon_j=\varepsilon_{-1}$, and there is even $l_j$ such that $\varepsilon_j=-1$. Again, we may assume that $j=-2$. In this case, $f_{-2}=-a$ is odd, where $a=(q^{l_{-2}}+1)/2$. Set $b=(q^{l_{-1}}-\varepsilon_{-1})/2$. Observe that $(a,b)=1$ by Lemma \ref{nt}$(iii)$. Hence, we again apply Lemma \ref{mt}$(iv)$, where  $c=-a$, to bring (\ref{orm5}) to the form (\ref{orm6}),
where $g_i=2af_i$. The rest of the argument is as above, except that the application of  Lemma \ref{mt}$(iii)$ in the end is justified due to
$$
(2ab)_2=\frac{1}{2}(q^{l_{-2}}+1)_2(q^{l_{-1}}-\varepsilon_{-1})_2=(q-\varepsilon_{-1})_2 = (q-\varepsilon_i)_2=(q^{l_i}-\varepsilon_i)_2,
$$
since both $l_{-1}$ and $l_i$ are odd and $\varepsilon_{-1}=\varepsilon_i$ by assumption. This proves item $(ii)$ of the theorem.

{\em Case Ic.} Suppose that, for every odd $l_j$, we have $\varepsilon_j=\varepsilon_{-1}$ and, for every even $l_j$, we have  $\varepsilon_j=1$.

If $l_j$ is odd then $f_j=-a+b$, where $a=(q^{l_j}-\varepsilon_{-1})/2$ and $b=(q^{l_{-1}}-\varepsilon_{-1})/2$. By lemma \ref{nt}$(iv)$, we see that both $a/(a,b)$ and $b/(a,b)$ are odd. Hence, $2(a,b)$ divides $-a+b$ and  the submatrix
\begin{equation}\label{tm}
\left(
\begin{array}{rr}
2a & f_j \\
. & 2b
\end{array}
\right)
\end{equation}
of (\ref{orm5})
can be diagonalized by Lemma \ref{mt}$(iii)$ using the matrices $P$ and $Q$ given in (\ref{pq3}).
Observe that due to the form of $Q$ this diagonalization does not affect the remaining entries of
(\ref{orm5}) and hence may be repeated for all odd $l_j$ independently.

Similarly, if $l_j$ is even then we have $f_j=-a$, where $a=(q^{l_j}-1)/2$. Set
$b=(q^{l_{-1}}-\varepsilon_{-1})/2$. By lemma \ref{nt}$(v)$, we have $2(a,b)=q^{(l_j,l_{-1})}-\varepsilon_{-1}$ which
divides $-a$, since $l_j/(l_j,l_{-1})$ is even. Again, applying Lemma \ref{mt}$(iii)$ to the submatrix (\ref{tm}) independently for all even $l_j$, we complete the diagonalization of (\ref{orm5}). This proves item $(iv)$ of the theorem in this case.

{\em Case II.} Assume that all $l_i$ are even and $\varepsilon_{-1}=-1$. Then $d=2n$ and
$f_i=-n(q^{l_i}-\varepsilon_i)$ is a multiple of the diagonal pivot in the $i$th row. Obvious column transformations annihilate all $f_i$ and bring (\ref{orm5}) to a diagonal form. This completes the proof of item $(iv)$ of the theorem.

\begin{table}[htb]
\caption{Maximal tori in $\operatorname{Spin}^\pm_8(q)$ \label{sp8}}
\begin{tabular}{ll|ll}
\hline
\multicolumn{2}{c|}{$\operatorname{Spin}^+_8(q)$} & \multicolumn{2}{c}{$\operatorname{Spin}^-_8(q) \vphantom{A^{A^A}}$}\\
\hline
$[1,1,1,1]$ & $\mathbb{Z}_{q-1}\times\mathbb{Z}_{q-1}\times\mathbb{Z}_{q-1}\times\mathbb{Z}_{q-1}$ & $[1,1,1,\overline{1}\,]$ & $\mathbb{Z}_{q^2-1}\times\mathbb{Z}_{q-1}\times\mathbb{Z}_{q-1} \vphantom{A^{A^A}}$\\
$[1,1,\overline{1},\overline{1}\,]$ & $\mathbb{Z}_{q^2-1}\times\mathbb{Z}_{q+1}\times\mathbb{Z}_{q-1}$ & $[1,\overline{1},\overline{1},\overline{1}\,]$ & $\mathbb{Z}_{q^2-1}\times\mathbb{Z}_{q+1}\times\mathbb{Z}_{q+1}$\\
$[\,\overline{1},\overline{1},\overline{1},\overline{1}\,]$ & $\mathbb{Z}_{q+1}\times\mathbb{Z}_{q+1}\times\mathbb{Z}_{q+1}\times\mathbb{Z}_{q+1}$ & $[2,1,\overline{1}\,]$ & $\mathbb{Z}_{q^2-1}\times\mathbb{Z}_{q^2-1}$\\
$[2,1,1]$ & $\mathbb{Z}_{q^2-1}\times\mathbb{Z}_{q-1}\times\mathbb{Z}_{q-1}$ & $[1,1,\overline{2}\,]$ & $\mathbb{Z}_{(q^2+1)(q-1)}\times\mathbb{Z}_{q-1}$\\
$[2,\overline{1},\overline{1}\,]$ & $\mathbb{Z}_{q^2-1}\times\mathbb{Z}_{q+1}\times\mathbb{Z}_{q+1}$ & $[\overline{2},\overline{1},\overline{1}\,]$ & $\mathbb{Z}_{(q^2+1)(q+1)}\times\mathbb{Z}_{q+1}$\\
$[1,\overline{2},\overline{1}\,]$ & $\mathbb{Z}_{q^2+1}\times\mathbb{Z}_{q^2-1}$ & $[2,\overline{2}\,]$ & $\mathbb{Z}_{q^2+1}\times\mathbb{Z}_{q^2-1}$\\
$[2,2]^\pm$ & $\mathbb{Z}_{q^2-1}\times\mathbb{Z}_{q+1}\times\mathbb{Z}_{q+1}$ & $[1,\overline{3}\,]$ & $\mathbb{Z}_{(q^3+1)(q-1)}$\\
$[\overline{2},\overline{2}\,]$ & $\mathbb{Z}_{q^2+1}\times\mathbb{Z}_{q^2+1}$ & $[3,\overline{1}\,]$ & $\mathbb{Z}_{(q^3-1)(q+1)}$\\
$[3,1]$ & $\mathbb{Z}_{q^3-1}\times\mathbb{Z}_{q-1}$ & $[\,\overline{4}\,]$ & $\mathbb{Z}_{(q^4+1)}$\\
$[\,\overline{3},\overline{1}\,]$ & $\mathbb{Z}_{q^3+1}\times\mathbb{Z}_{q+1}$ \\
$[4]^\pm$ & $\mathbb{Z}_{q^2+1}\times\mathbb{Z}_{q^2-1} \vphantom{q_{q_{q_q}}}$\\
\hline
\end{tabular}
\end{table}

{\em Acknowledgement.} The author is thankful to Dr. A.\,Buturlakin for a useful discussion related to this paper.

\makeatletter
\renewcommand*{\@biblabel}[1]{\hfill\bf#1.}
\makeatother

\end{document}